\documentclass[12pt]{amsart}

\usepackage[hmargin=0.8in,height=8.6in]{geometry}
\usepackage{amssymb,amsthm, times}
\usepackage{delarray,verbatim}
\usepackage{srcltx}
\usepackage{ifpdf}
\ifpdf
\usepackage[pdftex]{graphicx}
 \else
\usepackage[dvips]{graphicx}
 \fi

\usepackage{ifpdf}
\usepackage{color}
\definecolor{webgreen}{rgb}{0,.5,0}
\definecolor{webbrown}{rgb}{.8,0,0}
\definecolor{emphcolor}{rgb}{0.95,0.95,0.95}

\usepackage{hyperref}
\hypersetup{%
          colorlinks=true,
          linkcolor=webbrown,
          filecolor=webbrown,
          citecolor=webgreen,
          breaklinks=true}
\ifpdf \hypersetup{pdftex,
            pdfstartview=FitH, 
            bookmarksopen=true,
            bookmarksnumbered=true
} \else \hypersetup{dvips} \fi

\numberwithin{equation}{section}

\newtheorem{proposition}{Proposition}[section]
\newtheorem{corollary}{Corollary}[section]
\newtheorem{remark}{Remark}[section]
\newtheorem{lemma}{Lemma}[section]

\numberwithin{remark}{section} \numberwithin{proposition}{section}
\numberwithin{corollary}{section}

\renewcommand{\S}{\mathcal{S}}

\newcommand {\ME}{\mathbb{E}^{x}}

\newcommand {\R}{\mathbb{R}}

\newcommand {\F}{\mathcal{F}}
\newcommand {\A}{\mathcal{A}}

\newcommand {\p}{\mathbb{P}}

\newcommand {\E}{\mathbb{E}}

\newcommand{\diff}{{\rm d}}

\newcommand{\word}{\hspace{0.2cm}}
\newcommand{\conn}{\quad\text{and}\quad}
\newcommand{\1}{\mbox{1}\hspace{-0.25em}\mbox{l}}
\newcommand{\lev}{L\'{e}vy }

\newcommand{\s}{\bar{s}}
\newcommand{\II}{\mathcal{I}}

\title[Optimal Stopping of Diffusion and its Maximum with Boundary]{Explicit Solutions for Optimal Stopping of Maximum Process with Absorbing Boundary that Varies with It}
\author[M. Egami]{Masahiko Egami}
\address[M. Egami]{Graduate School of Economics,
Kyoto University, Sakyo-Ku, Kyoto, 606-8501, Japan}
\email{egami@econ.kyoto-u.ac.jp}
\urladdr{http://www.econ.kyoto-u.ac.jp/{\textasciitilde}egami/}
\thanks{First Draft: September 28, 2015.  This version: January 17, 2016. This work is in part supported
by Grant-in-Aid for Scientific Research (B) No. 26285069, Japan Society for the Promotion of Science.}
\author[T. Oryu]{Tadao Oryu}
\address[T. Oryu]{Institute of Economic Research,
Kyoto University, Sakyo-Ku, Kyoto, 606-8501, Japan}
\email{oryu@kier.kyoto-u.ac.jp}

\date{}
\begin{document}
\begin{abstract}
We provide, in a general setting, explicit solutions for optimal stopping problems that involve a diffusion process and its running maximum.  Besides, a new feature includes absorbing boundaries that vary with the value of the running maximum. The existence of the absorbing boundary of this type makes the problem
harder but more practical and flexible.
Our approach is to use the excursion theory for \lev processes. Since general diffusions are, in particular, not of independent increments,  we use an appropriate measure change to make the process have that property. Then we rewrite the original two-dimensional problem as an infinite number of one-dimensional ones and complete the solution. We show general solution methods with explicit value functions and corresponding optimal strategies, illustrating them by some examples.
\end{abstract}

\maketitle \noindent \small{\textbf{Key words:} Optimal stopping; excursion theory;
 diffusions; scale functions.\\
\noindent Mathematics Subject Classification (2010) : Primary: 60G40
Secondary: 60J75 }\\

\section{Introduction}
We let  $X=(X_t, t\ge 0)$ be one-dimensional diffusion and denote by $Y$ the reflected process,
  \[
  Y_t=S_t -X_t
  \]
  where $S_t=\sup _{u \in [0,t]} X_u\vee s$ with $s=S_0$. Hence $Y$ is the excursion of $X$ from its running maximum $S$.  We consider an optimal stopping problem that involves both $X$ and $S$.  It is subject to absorbing boundary that varies with $S$.  That is,
  \begin{eqnarray*}
\bar{V}(x,s)&=&\sup_{\tau} \E^{x,s} \left[ \int^{\tau\wedge \zeta}_0 e^{-qt}f(X_t,S_t)\diff t+ \1_{\{\tau<\zeta\}}e^{-q\tau}g(X_{\tau},S_{\tau}) \right]
\end{eqnarray*}
subject to  absorption \[\zeta:=\inf \{t\geq 0 : S_t-X_t > b(S_t)\}.\]
where the rewards $f$ and $g$ are measurable functions from $\R^2$ to $\R_+$ and $b: \R\mapsto \R_+$ is also a measurable function. The rigorous mathematical definition of this problem is presented in Section \ref{sec:model}.    This setup means that while $X$ grows and keeps attaining new maxima, the absorbing boundary is accompanying with $S$.
In this paper, we shall solve for optimal strategy and corresponding value function explicitly along with optimal stopping region in the $(x, s)$-plane. The existence of the absorbing state makes the stopping region more complex.

The idea is the following: we look at excursions that occur from each level of $S$, and reduce the problem to an infinite number of one-dimensional optimal stopping problems.  For finding the explicit form of the value function, we employ the theory of excursion of \lev processes, in particular the characteristic measure that is related to the height of excursions.  (Refer to Bertoin \cite{Bertoin_1996} as a general reference.)
 Since the diffusion $X$ is not in general of independent increments, we use the measure change \eqref{eq:change-of-measure} to make the diffusion behave like a Brownian motion under the new measure.  Having done that, we solve, at each level of $S$, one-dimensional optimal stopping problems by using the excessive characterization of the value function.  This corresponds to the concavity of the value function after certain transformation.  See Dynkin\cite{dynkin}, Alvarz\cite{alvarez2} and Dayanik and Karatzas \cite{DK2003}.  Note that for the excursion theory for spectrally negative \lev processes (that have only downward jumps), see also Avram et al. \cite{avram-et-al-2004}, Pistorius \cite{Pistorius_2004} \cite{Pistorius_2007} and Doney \cite{Doney_2005} where, among others, an exit problem of the reflected process $Y$ is  studied.
For optimal stopping that involve both $S$ and $X$, we mention a pioneering work of Peskir\cite{peskir1998}.   There are also  Ott \cite{ott_2013} and Guo and Zervos \cite{Guo-Zervos_2010}.  In the former paper, the author solves problems including a capped version of the Shepp-Shiryaev problem \cite{shepp-shiryaev-1993} and the latter makes  another contribution that extends \cite{shepp-shiryaev-1993}. A recent development in this area includes Alvarez and Mato\"{a}ki \cite{alvarez-matomaki2014} where a discretized approach is taken to find optimal solutions and a numerical algorithm is presented.

Our contributions in this paper may advance the literature in several respects:  we do not assume any specific forms or properties in the reward functions and we provide \emph{explicit} forms of the value function \emph{with or without the absorbing boundary} and illustrate the procedure of the solution method. Hence we present a very general solution method in a general setting.

The existence of the absorbing state that varies with the maximum process leads to  various applications.  In this paper, we provide a new problem where an investor puts her money in risky assets and she maintains the following investment policy:  if the drawdown  of the asset value exceeds certain level, say a fraction of the running maximum, she would sell all her risky assets and put the proceeds to risk-free assets. This investment policy is due to avoid the liquidity problem that prevailed in the years of the financial crisis. That is, when asset markets deteriorate, some investors are forced to sell their assets further since they need cash to repay their debt.  That would lead to vicious circle: further price depreciations and depletion of liquidity.  Hence her problem is to set up a rule as to when she converts her risky assets to risk-free.  Other applications to real-life problems include bank's failures during the financial crisis.  These banks had maintained high leverage and accordingly, they were, despite their large size,  not so safe since the bankruptcy threshold (=absorption state)  keeps up with the size of the banks.  Egami and Oryu \cite{Egami-Oryu2013} \cite{Egami-Oryu2013a} modeled this phenomenon by using spectrally negative \lev processes for the bank's asset size $X$.
Further applications are possible:  one can add  barriers that would nullify the value of lookback options (with the terminal payoff $g(x, s)=s-x$ for instance).

The rest of the paper is organized as follows.  In Section
\ref{sec:model}, we formulate a mathematical model with a review of some important facts of linear diffusions, and then
find an optimal solution in Section \ref{sec:solution}.  Finally in Section 4, we shall demonstrate the solution methodology using an example of the investment problem in risky assets, providing an explicit calculation.

\section{Mathematical Model}\label{sec:model} 

Let the diffusion process $X=\{X_t;t\geq 0\}$ represent the state variable defined on the probability space
$(\Omega, \F, \p)$, where $\Omega$ is the set of all possible realizations of the
stochastic economy, and $\p$ is a probability measure defined on $\F$. The state space of $X$ is given by
$\mathcal{I}\subseteq \mathbb{R}$.  We denote by
$\mathbb{F}=\{\F_t\}_{t\ge 0}$ the filtration with respect to which $X$ is adapted and with the usual
conditions being satisfied. We assume that $X$ satisfies the following stochastic differential equation:
\[
\diff X_t=\mu (X_t)\diff t + \sigma (X_t) \diff B_t,\quad X_0=x,
\]
where $\mu (x), \sigma(x) \in \R$ for any $x\in \R$ and $B=\{B_t: t\ge 0\}$ is a standard Brownian motion.

The running maximum process $S=\{S_t;t\geq 0\}$ with $s=S_0$ is defined by
$S_t=\sup _{u \in [0,t]} X_u\vee s$. In addition, we write $Y$ for the reflected process defined by $Y_t=S_t-X_t$,
and let $\zeta$ be the stopping time defined by
\[\zeta:=\inf \{t\geq 0 : S_t-X_t > b(S_t)\},\]
which is the time of absorption.  Note that $b: \R\mapsto \R_+$ is a measurable function and  the level at which the process $X$ is absorbed depends on $S$.

We consider the following optimal stopping problem and the value function $\bar{V}:\R^2 \mapsto \R$ associated with initial values $X_0=x$ and $S_0=s$;
\begin{eqnarray}\label{problem}
\bar{V}(x,s)&=&\sup_{\tau\in\S} \E^{x,s} \left[ \int^{\tau\wedge \zeta}_0 e^{-qt}f(X_t,S_t)\diff t+ \1_{\{\tau<\zeta\}}e^{-q\tau}g(X_{\tau},S_{\tau}) \right]
\end{eqnarray}
where $\mathbb{P}^{x,s}(\,\cdot\,):=\mathbb{P}(\,\cdot\,|\,X_0=x, S_0=s)$ and $\E^{x,s}$ is the expectation operator corresponding to $\mathbb{P}^{x,s}$, $q\geq 0$ is the constant discount rate and $\S$ is the set of all $\mathbb{F}$-adapted stopping times.  The payoff is composed of two parts; the running income to be received continuously until stopped or absorbed, and the terminal reward part.  The running income function
$f:\R^2 \mapsto \R$ is a measurable function that satisfies
\[
\E^{x,s}\left[\int_0^\infty e^{-qt}|f(X_t,S_t)|\diff t\right]<\infty.
\]
The reward function $g:\R^2 \mapsto \R_+$ is assumed to be measurable. Our main purpose is to calculate $\bar{V}$ and to find the stopping time $\tau^*$ which attains the supremum.

For each collection $D=(D(s))_{s\in\R}$ of Borel measurable sets, we define a stopping time $\tau(D)$ by
\begin{equation}\label{eq:tau-D}
\tau(D):=\inf\{t\geq 0:S_t-X_t\in D(S_t)\},
\end{equation}
and define a set of stopping times $\S'$ by $\S':=\{\tau(D)\}$. In other words, $\tau(D)$ is the first time the excursion $S-X$ from level, say $S=s$, enters the region $D(s)$. Suppose that $D(m) = (c,b(m)]$ for any $m\in\R$, we write
\[
\tau_c:=\inf\{t\geq 0:S_t-X_t>c\}.
\]

Next we let $\S'(b)$ be the set of stopping times defined by
\begin{equation}\label{eq:D(m)}
\S'(b):=\{\tau(D):D(m)\subset [0,b(m)] \;{\rm for\; all }\;m\in \R\}.
\end{equation}
Note that if $\tau\in\S'(b)$, then $\tau < \zeta$. The following lemma shows that it suffices to consider stopping times $\tau\in\S'(b)$:
\begin{lemma}\label{lem:reduction}
Let us define $u: \R^2\times \S' \mapsto \R$ by
\begin{eqnarray*}
u(x, s; \tau)&:=&\E^{x, s}\left[ \int^{\tau\wedge \zeta}_0 e^{-qt}f(X_t,S_t)\diff t+ \1_{\{\tau<\zeta\}}e^{-q\tau}g(X_{\tau},S_{\tau})\right]
\end{eqnarray*}
Then for any $\tau\in \S'$, we can find a $\tau'\in \S'(b)$ such that $u(x, s; \tau)=u(x, s; \tau')$.
\end{lemma}
\begin{proof}
Set the collections $D_i=(D_i(s))_{s\in\R}$, $i=1,2$, by $D_2(s)=D_1(s)\cap [0,b(s)]$ for some $s$ and $D_1(m)=D_2(m)$ on $m\neq s$.
Then it is clear from the definition that $\tau(D_1)<\zeta$ if and only if $\tau(D_2)<\zeta$, and $\tau(D_1)=\tau(D_2)$ on $\{\tau(D_1)<\zeta\}$. Hence the right hand side of
(\ref{problem}) for $\tau=\tau(D_1)$ and $\tau=\tau(D_2)$ are equal to each other. Hence the lemma is proved.
\end{proof}

Due to the above lemma, we can reduce the original problem \eqref{problem} to the following:
\begin{equation}\label{problem2}
\bar{V}(x,s)=\sup_{\tau\in\S'(b)} \E^{x,s} \left[ \int^{\tau}_0 e^{-qt}f(X_t,S_t)\diff t+ e^{-q\tau}g(X_{\tau},S_{\tau})\right]
\end{equation}

\subsection{Optimal Strategy}\label{subsec:optimality}  
We will reduce the problem \eqref{problem2} to an infinite number of one-dimensional optimal stopping problem and discuss the optimality of the
proposed strategy \eqref{eq:tau-D}.  Let us denote by $\bar{f}:\R^2\mapsto\R$ the $q$-potential of $f$:
\[
\bar{f}(x,s):=\E^{x,s} \left[ \int^{\infty}_0 e^{-qt}f(X_t,S_t)\diff t \right].
\]
From the strong Markov property of $(X,S)$, we have
\begin{eqnarray}\label{eq:potential-rewrite}
&&\E^{x,s} \left[ \int^{\tau\wedge \zeta}_0 e^{-qt}f(X_t,S_t)\diff t \right]\\
&=&\E^{x,s} \left[\int^{\infty}_0 e^{-qt}f(X_t,S_t)\diff t - \int^{\infty}_{\tau\wedge \zeta} e^{-qt}f(X_t,S_t)\diff t \right]\nonumber\\
&=&\bar{f}(x,s)-\E^{x,s} \left[\E \left[\int^{\infty}_{\tau\wedge \zeta} e^{-qt}f(X_t,S_t)\diff t \Bigm| \mathcal{F}_{\tau\wedge \zeta}\right] \right]\nonumber\\
&=&\bar{f}(x,s)-\E^{x,s} \left[e^{-q(\tau\wedge \zeta)}\E^{X_{\tau\wedge \zeta},S_{\tau\wedge \zeta}}\left[\int^{\infty}_0 e^{-qt}f(X_t,S_t)\diff t \right] \right]\nonumber\\
&=&\bar{f}(x,s)-\E^{x,s} \left[e^{-q(\tau\wedge \zeta)}\bar{f}(X_{\tau\wedge \zeta},S_{\tau\wedge \zeta})\right]\nonumber\\
&=&\bar{f}(x,s)-\E^{x,s} \left[\1_{\{\tau<\zeta\}}e^{-q\tau}\bar{f}(X_{\tau},S_{\tau})+\1_{\{\zeta\le \tau\}}e^{-q\zeta}\bar{f}(X_{\zeta},S_{\zeta})\right].\nonumber
\end{eqnarray}
Hence the value function $\bar{V}$ can be written as
\begin{eqnarray}\label{eq:value-function-rewritten}
\bar{V}(x,s)&=&\bar{f}(x,s)+V(x,s),\nonumber
\end{eqnarray}
where
\begin{align}\label{eq:V-modified}
\hspace{0.5cm} V(x,s):=\sup_{\tau\in\S}\E^{x,s}\left[\1_{\{\tau<\zeta\}}e^{-q\tau}(g-\bar{f})(X_{\tau},S_{\tau})-\1_{\{\zeta\le \tau\}}e^{-q\zeta}\bar{f}(X_{\zeta},S_{\zeta})\right].
\end{align}
Since $\bar{f}(x, s)$ has nothing to do with the choice of $\tau$, we concentrate on $V(x, s)$.

Let us first define the first passage times of $X$:
\[T_a:=\inf\{t\geq 0:X_t > a\}\quad\text{and}\quad T_a^{-}:=\inf\{t\ge 0: X_t
 <a\}.\]
By the dynamic programming principle, we can write $V(x,s)$ as
\begin{align}\label{eq:dp}
V(x,s)
&=\sup_{\tau\in\S}\E^{x,s}\left[\1_{\{\tau<\theta\}}\1_{\{\tau<\zeta\}}e^{-q\tau}(g-\bar{f})(X_{\tau},S_{\tau})\right.\\
&\left.-\1_{\{\zeta<\theta\}}\1_{\{\zeta\leq\tau\}}e^{-q\zeta}\bar{f}(X_{\zeta},S_{\zeta})+\1_{\{\theta<\tau\wedge\zeta\}}e^{-q\theta}V(X_{\theta},S_{\theta})\nonumber\right],
\end{align}
for any stopping time $\theta\in\S$. See, for example, Pham \cite{Pham-book} page 97. Now we set $\theta =T_s$ in (\ref{eq:dp}).
For each level $S=s$ from which an excursion $Y=S-X$ occurs, the value $S$ does not change during the excursion.
Hence, during the first excursion interval  from $S_0=s$, we have $\zeta=T^-_{s-b(s)}$ and $S_t=s$ for any $t\leq T_s$, and (\ref{eq:dp}) can be written as the following one-dimensional problem for the state process $X$;
\begin{eqnarray}\label{eq:one-dim-version}
V(x,s)&=&\sup_{\tau\in\S}\E^{x,s}\left[\1_{\{\tau<T_s\}}\1_{\{\tau<T^-_{s-b(s)}\}}e^{-q\tau}(g-\bar{f})(X_{\tau},s)\right.\\
&&-\1_{\{T^-_{s-b(s)}<T_s\}}\1_{\{T^-_{s-b(s)}\leq \tau\}}e^{-qT^-_{s-b}}\bar{f}(X_{T^-_{s-b(s)}},s)\nonumber\\
&&\left.+\1_{\{T_s<\tau\wedge T^-_{s-b(s)}\}}e^{-qT_s}V(s,s)\right].\nonumber
\end{eqnarray}
Now we can look at \emph{only} the process $X$ and find $\tau^*\in \S$.

In relation to \eqref{eq:one-dim-version}, we consider the following one-dimensional optimal stopping problem as for $X$ and its value function $\widehat{V}:\R^2\mapsto \R$;
\begin{align}\label{eq:Vhat}
\widehat{V}(x,s)&=\sup_{\tau\in\S}\E^{x,s}\left[\1_{\{\tau<T_s\}}\1_{\{\tau<T^-_{s-b(s)}\}}e^{-q\tau}(g-\bar{f})(X_{\tau},s)\right.\\
&\left.-\1_{\{T^-_{s-b(s)}<T_s\}}\1_{\{T^-_{s-b(s)}\leq \tau\}}e^{-qT^-_{s-b}}\bar{f}(X_{T^-_{s-b(s)}},s)+\1_{\{T_s<\tau\wedge T^-_{s-b(s)}\}}e^{-qT_s}K\right]\nonumber,
\end{align}
where $K\geq 0$ is a constant. Note that $V=\widehat{V}$ holds when $K=V(s,s)$, and $V(s,s)$ can be obtained by our solution method offered in Section 3.

Recall that, in the pursuit of optimal strategy $\tau^*$ in the linear diffusion case, we can utilize the full characterization of the value function and  optimal stopping
rule: an optimal stopping rule  is given by the threshold strategy in a very general setup.  In our present problem, optimal strategy belongs to the set of $\tau(D)$ in \eqref{eq:tau-D}.
See Dayanik and Karatzas \cite{DK2003}; Propositions 5.7 and 5.14. See also Pham \cite{Pham-book}; Section 5.2.3.  Note that, however, writing the value of $V(s, s)$ in an explicit form is not trivial and is an essential part of the solution, which we shall do in the next section (see Propositions \ref{prop:1} and \ref{prop:2}).

\subsection{Important Facts of Diffusions}\label{sec:diffusion-facts}
Let us recall the fundamental facts about one-dimensional diffusions; let the differential operator $\A$ be the infinitesimal generator of the process $X$ defined by
\[
\A v(\cdot)=\frac{1}{2}\sigma^2(\cdot)\frac{\diff^2  v}{\diff x^2}(\cdot)+\mu(\cdot)\frac{\diff v}{\diff x}(\cdot)
\]
and consider the ODE $(\A-q)v(x)=0$. This equation has two
fundamental solutions: $\psi(\cdot)$ and $\varphi(\cdot)$.
We set
$\psi(\cdot)$ to be the increasing and $\varphi(\cdot)$ to be the
decreasing solution.  They are linearly independent positive
solutions and uniquely determined up to multiplication. It is well
known that
\begin{align*}
  \ME[e^{-\alpha\tau_z}]
  =\begin{cases}
\frac{\psi(x)}{\psi(z)}, & x \le z,\\[4pt]
\frac{\varphi(x)}{\varphi(z)}, &x \ge z.
  \end{cases}
\end{align*} For the complete characterization of $\psi(\cdot)$ and
$\varphi(\cdot)$, refer to It\^{o} and McKean \cite{IM1974}. Let us now define
\begin{align} \label{eq:F}
F(x)&:=\frac{\psi(x)}{\varphi(x)}, \hspace{0.5cm} x\in
\mathcal{I}.
\end{align}
Then $F(\cdot)$ is continuous and strictly increasing.  Next,
following Dynkin (pp.\ 238, \cite{dynkin}), we define concavity
of a function with respect $F$ as follows:
A real-valued function $u$ is called \emph{$F$-concave} on $\mathcal{I}$
if, for every $x\in[l, r]\subseteq \mathcal{I}$,
\[
u(x)\geq
u(l)\frac{F(r)-F(x)}{F(r)-F(l)}+u(r)\frac{F(x)-F(l)}{F(r)-F(l)}.\]

Now consider the optimal stopping problem:
\[V(x)=\sup_{\tau\in \S}\E^x[e^{-q \tau}h(X_\tau)]
\]
where $h$: $[c, d]\mapsto \R_+$. Let $W(\cdot)$ be the smallest nonnegative concave majorant of $H:=(h/\varphi)\circ F^{-1}$ on $[F(c), F(d)]$ where
$F^{-1}$ is the inverse of $F$. Then we have $V(x)=\varphi(x)W(F(x))$ and the optimal stopping region are
\[
\Gamma:=\{x\in [c, d]: V(x)=h(x)\} \conn \tau^*:=\inf\{t\ge 0: X_t\in \Gamma\}
\] as in Propositions 4.3 and 4.4 of \cite{DK2003}.

\section{Explicit Solution}\label{sec:solution}
Now we look to an explicit solution of $\bar{V}$ for $\tau\in\S'$. The first step is to find $V(s, s)$ in \eqref{eq:one-dim-version}.
\subsection{When \bf{$S_0=X_0$}}
As a first step, we consider the case $S_0=X_0$. Set stopping times
$T_m$ as $T_m=\inf\{t\geq 0:X_t > m\}$ and recall $\S'(b):=\{\tau(D):D(m)\subset [0,b(m)]\}$ as in \eqref{eq:D(m)}. Define the function $l_D:\R_+\mapsto\R_+$ by
\begin{equation}\label{eq:lD}
l_D(m):=\inf D(m).
\end{equation}
for which $\tau(D) \in \S'(b)$.   In other words, if $S_0=X_0$, given a threshold strategy $\tau=\tau(D)$ where  $D(m)$ is in the form of $[a, c]\subset [0, b(m)]$, the value $l_{D(m)}$ is equal to $a$.  Accordingly,
$S_{\tau_{l_D(m)}}=S_{\tau_a}$ on the set $\{S_{\tau_{l_D(m)}}\in\diff m\}$.

 From the strong Markov property of $(X,S)$, when $\tau(D)\in\S'(b)$ and
$S_0=X_0=s$, we have
\begin{eqnarray}\label{eq:interim}
&&\E^{s,s}\left[\1_{\{\tau(D)<\zeta\}}e^{-q\tau(D)}(g-\bar{f})(X_{\tau(D)},S_{\tau(D)})\right]\\
&=&\int^{\infty}_s\E^{s,s}\left[\1_{\{\tau(D)<\zeta, S_{\tau(D)}\in\diff m\}}e^{-q\tau(D)}(g-\bar{f})(X_{\tau(D)},S_{\tau(D)})\right]\nonumber\\
&=&\int^{\infty}_s\E^{s,s}\left[\1_{\{T_m\leq\tau(D)\}}e^{-qT_m}\E^{m,m}\left[e^{-q\tau_{l_D(m)}}(g-\bar{f})(X_{\tau_{l_D(m)}},S_{\tau_{l_D(m)}})\right.\right.\nonumber\\
&&\left.\left.\times\1_{\{S_{\tau_{l_D(m)}}-X_{\tau_{l_D(m)}}\leq b(m), S_{\tau_{l_D(m)}}\in \diff m\}}\right]\right]\nonumber\\
&=&\int^{\infty}_s\E^{s,s}\left[\1_{\{S_{\tau(D)}\geq m\}}e^{-qT_m}\right](g-\bar{f})(m-l_D(m),m)\nonumber\\
&&\times\E^{m,m}\left[e^{-q\tau_{l_D(m)}}\1_{\{S_{\tau_{l_D(m)}}\in \diff m\}}\right]\nonumber.
\end{eqnarray}
Now we calculate these expectations by changing probability measure.
We introduce the probability measure $\p^{\varphi}_{x,s}$ defined by
\begin{equation}\label{eq:change-of-measure}
\p^{\varphi}_{x,s}(A):=\frac{1}{\varphi(x)}\E^{x,s}\left[ e^{-qt}\varphi(X_t)\1_{A} \right], \text{for every } A\in\F.
\end{equation}
Then $F(X)=(F(X_t))_{t\in \R_+}$ is a process in natural scale under this measure $\p^{\varphi}_{x,s}$. See Borodin and Salminen \cite{borodina-salminen}(pp.\ 33) and Dayanik and Karatzas \cite{DK2003} for detailed explanations. Hence we can write, under $\p^{\varphi}_{x,s}$, $F(X_t)=\sigma_F B^{\varphi}_t$, where $\sigma_F>0$ is a constant and $B^{\varphi}$ is a Brownian motion under $\p^{\varphi}_{x,s}$. Since $F(X)$ is a \lev process, we can define the process $\eta:=\{\eta_t;t\geq 0\}$ of the height of the excursion as
\[
\eta_u:= \sup\{(S-X)_{T_{u-}+w} : 0\leq w \leq T_u - T_{u-}\}, \text{ if \;}  T_u > T_{u-},
\]
and $\eta_u=0$ otherwise, where $T_{u-}:=\inf\{t\ge0:X_t\ge u\}=\lim_{m\rightarrow u-}T_{m}$.
Then $\eta$ is a Poisson point process, and we denote its
characteristic measure under $\p^{\varphi}_{x,s}$ by $\nu:\F\mapsto\R_+$ of $F(X)$.
It is well known that
\[
\nu[u,\infty)=\frac{1}{u}, \text{ for } u\in\R_+.
\] See, for example, \c{C}inlar \cite{cinlar} (pp.\ 416).
By using these notations, we have\footnote{Note that when the diffusion $X$ is a standard Brownian motion $B$, then $F(x)=x$ and the right-hand side reduces to $\exp\left(-\int_s^m\frac{\diff u}{l_D(u)}\right)$.}
\begin{eqnarray}\label{eq:probability}
\p^{\varphi}_{s,s}(S_{\tau(D)}>m)&=&\exp\left( -\int^{F(m)}_{F(s)}\nu[y-F(F^{-1}(y)-l_D(F^{-1}(y))),\infty)\diff y \right)\nonumber\\
&=&\exp\left( -\int^m_{s}\frac{F'(u)\diff u}{F(u)-F(u-l_D(u))} \right).
\end{eqnarray}
On the other hand, from the definition of the measure $\p^{\varphi}_{x,s}$, we have
\begin{eqnarray*}
\p^{\varphi}_{s,s}(S_{\tau(D)}>m)
&=&\frac{1}{\varphi(s)}\E^{s,s}\left[ e^{-qT_m}\varphi(X_{T_m})\1_{\{ S_{\tau(D)}>m \}} \right]\\
&=&\frac{\varphi(m)}{\varphi(s)}\E^{s,s}\left[ e^{-qT_m}\1_{\{ S_{\tau(D)}>m \}} \right].
\end{eqnarray*}
Combining these two things together,
\begin{equation}
\E^{s,s}\left[ e^{-qT_m}\1_{\{ S_{\tau(D)}>m \}} \right]=\frac{\varphi(s)}{\varphi(m)}\exp\left( -\int^m_s\frac{F'(u)\diff u}{F(u)-F(u-l_D(u))} \right).
\end{equation}
Similarly, by changing the measure and noting that $X_{\tau_{l_D(m)}}=m-l_D(m)$, we have
\begin{eqnarray}
&&\E^{m,m}\left[e^{-q\tau_{l_D(m)}}\1_{\{S_{\tau_{l_D(m)}}\in \diff m\}}\right]\\ \nonumber
&=&\frac{\varphi(m)}{\varphi(m-l_D(m))}\cdot\frac{1}{\varphi(m)}\E^{m,m}\left[e^{-q\tau_{l_D(m)}}\varphi(X_{\tau_{l_D(m)}})\1_{\{S_{\tau_{l_D(m)}}\in \diff m\}}\right]\\ \nonumber
&=&\frac{\varphi(m)}{\varphi(m-l_D(m))}\p^{\varphi}_{m,m}(F(S_{\tau_{l_D(m)}})\in\diff F(m))\\ \nonumber
&=&\frac{\varphi(m)}{\varphi(m-l_D(m))}\p^{\varphi}_{m,m}(F(S_{\tau_{l_D(m)}})-F(X_{\tau_{l_D(m)}})=l_D(m), F(S_{\tau_{l_D(m)}})\in\diff F(m))\nonumber.
\end{eqnarray}

Since $F(X)$ is a \lev process under  $\p^{\varphi}_{s,s}$, we can apply Theorem 2 in Pistorius \cite{Pistorius_2007} to calculate the last probability. Then we have
\begin{equation}
\p^{\varphi}_{m,m}(F(S_{\tau_{l_D(m)}})-F(X_{\tau_{l_D(m)}})=l_D(m), F(S_{\tau_{l_D(m)}})\in\diff F(m))=\frac{F'(m)\diff m}{F(m)-F(m-l_D(m))}.
\end{equation}

Thanks to Lemma \ref{lem:reduction}, we have, up to this point, proved the following:
\begin{proposition}\label{prop:1}
When $S_0=X_0$, the function $V(s,s)$ for $\tau\in \S'$ can be represented by
\begin{eqnarray}\label{eq:V(s,s) integral form}
V(s,s)&=&\sup_{l_D}\int^\infty_s\frac{\varphi(s)}{\varphi(m-l_D(m))}\exp\left( -\int^m_s\frac{F'(u)\diff u}{F(u)-F(u-l_D(u))} \right)\\
&&\times\frac{F'(m)(g-\bar{f})(m-l_D(m),m)}{F(m)-F(m-l_D(m))}\diff m. \nonumber
\end{eqnarray}
\end{proposition}
This proposition applies to general cases.  The following corollary can be shown directly from the above integral:
\begin{corollary}\label{cor:1}
If $(g-\bar{f})(x, s)$ does not depend on $s$, the function $V(s, s)$ reduces to
\begin{equation}\label{eq:simple-case-no-s}
  V(s, s)=\sup_{l_D(s)}\frac{\varphi(s)}{\varphi(s-l_D(s))}(g-\bar{f})(s-l_D(s),s),
\end{equation}
and $l^*_D(s)$ is the maximizer of the map $z\mapsto \frac{\varphi(s)}{\varphi(s-z)}
(g-\bar{f})(s-z, s)$.
\end{corollary}

We wish to
obtain more explicit formulae for $V(s,s)$.
For this purpose, let us denote\[
H(u; l_D):=\frac{F'(u)}{F(u)-F(u-l_D(u))}, \conn G(u; l_D):=(g-\bar{f})(u-l_D(u),u),
\]
to avoid the long expression and
rewrite \eqref{eq:V(s,s) integral form} in the following way:
for any $m-s>\epsilon > 0$,
\begin{eqnarray*}
V(s, s)&=&\sup_{l_D}\left[\int_{s}^{s+\epsilon}\frac{\varphi(s)}{\varphi(m-l_D(m))}\exp\left(-\int_s^{m}H(u; l_D)\diff u\right)H(m; l_D)G(m; l_D)\diff m\right.\\
&&+\left.\frac{\varphi(s)}{\varphi(s+\epsilon)}\exp\left(-\int_s^{s+\epsilon}H(u; l_D)\diff u\right)\right.\\
&&\hspace{0.8cm}\left.\times\int_{s+\epsilon}^\infty\frac{\varphi(s+\epsilon)}{\varphi(m-l_D(m))}\exp\left(-\int_{s+\epsilon}^{m}H(u; l_D)\diff u\right)H(m; l_D)G(m; l_D)\diff m\right]\\
&=&\sup_{l_D}\left[\int_{s}^{s+\epsilon}\frac{\varphi(s)}{\varphi(m-l_D(m))}\exp\left(-\int_s^{m}H(u; l_D)\diff u\right)H(m; l_D)G(m; l_D)\diff m\right.\\
&&+\left.\frac{\varphi(s)}{\varphi(s+\epsilon)}\exp\left(-\int_s^{s+\epsilon}H(u; l_D)\diff u \right) V(s+\epsilon, s+\epsilon)\right].
\end{eqnarray*}
This expression naturally motivates us to set $V_\epsilon : \R\mapsto\R$ as
\begin{equation}\label{eq:V-epsilon}
V_\epsilon(s):=\sup_{l_D(s)}\left[ \frac{\varphi(s)}{\varphi(s+\epsilon)}\exp\left(-\epsilon H(s; l_D)\right)V(s+\epsilon,s+\epsilon)+\frac{\varphi(s)}{\varphi(s-l_D(s))}\cdot\epsilon H(s; l_D)G(s; l_D)\right]
\end{equation} and we have $\lim_{\epsilon\downarrow 0}V_\epsilon (s)=V(s,s)$.
Dividing both sides by $\varphi(s)$ and choosing the optimal level $l_D^*(s)$,  we have, for any $s$ given,
\begin{equation}\label{eq:V-epsilon-in-F}
  \frac{V_\epsilon(s)}{\varphi(s)}=\left[\frac{V(s+\epsilon, s+\epsilon)}{\varphi(s+\epsilon)}e^{-\epsilon H(s; l^*_D)}+\frac{\epsilon H(s; l^*_D)G(s; l^*_D)}{\varphi(s-l^*_D(s))}\right].
\end{equation}
Let us consider the transformation of a Borel function $z$ defined on $-\infty\le c\le x\le d\le \infty$ through
\begin{equation}\label{eq:F-trans}
Z(y):=\frac{z}{\varphi}\circ F^{-1}(y)
\end{equation} on $[F(c), F(d)]$ where $F^{-1}$ is the inverse of the strictly increasing $F(\cdot)$ in \eqref{eq:F}.
 If we evaluate $Z$ at $y=F(x)$, we obtain $Z(F(x))=\frac{z(x)}{\varphi(x)}$, which is the form that appears in \eqref{eq:V-epsilon-in-F}. Note that
 \begin{equation}\label{eq:derivative}
  Z'(y)=q'(x) \quad \text{where}\word q'(x)=\frac{1}{F'(x)}\left(\frac{z}{\varphi}\right)'(x)
 \end{equation}

To make explicit calculations possible, we consider the case where the reward increases as $S$ does, a natural problem formulation. 

\begin{proposition}\label{prop:2}
Fix $s\in \mathcal{I}$.  If (1) the reward function $(g-\bar{f})(x, s)$ is nondecreasing in the second argument and (2) $\frac{\varphi(s)}{\varphi(s+\epsilon)}\frac{\varphi'(s+\epsilon)}{\varphi'(s)}<1$ for all $s\in \II$ and $\epsilon>0$, we have
  \begin{equation}\label{eq:V-explicit}
V(s,s)=\frac{\varphi(s)}{\varphi(s-l^*_D(s))}\cdot Q(s; l_D^*) \cdot (g-\bar{f})(s-l^*_D(s),s),
\end{equation}
where
\[
Q(u; l_D):=\frac{F'(u)\varphi'(u)}{\varphi''(u)[F(u)-F(u-l^*_D(u))]+F'(u)\varphi'(u)}
\]
and $l^*_D(s)$ is the maximizer of the map \[z \mapsto
\frac{\varphi(s)}{\varphi(s-z)}\cdot\frac{F'(s)\varphi'(s)}{\varphi''(s)[F(s)-F(s-z)]+F'(s)\varphi'(s)}\cdot(g-\bar{f})(s-z,s)\quad \text{on}\word [0,b].\]
\end{proposition}
Note that it can be confirmed that diffusions that satisfy the second assumption include geometric Brownian motion, Ornstein-Uhlenbeck process, etc.
\begin{proof}
First, we claim the following statement:
\begin{lemma}\label{lem:epsilon} Under the assumption of Proposition \ref{prop:2}, for $\epsilon>0$ sufficiently close to zero, we have
\begin{equation} \label{eq:V-V}
    \frac{V_\epsilon(s)}{\varphi(s)}=\alpha_s(\epsilon)\cdot\frac{V(s+\epsilon, s+\epsilon)}{\varphi(s+\epsilon)} \quad \word\text{where}\word
    \alpha_s(\epsilon):=\frac{\varphi'(s+\epsilon)}{\varphi'(s)}.
\end{equation}
\end{lemma}
\noindent Note that $\alpha_s(\epsilon)\in (0, 1)$ for all $s\in \mathcal{I}$ and $\epsilon>0$ and that $\alpha_s(\epsilon)\uparrow 1$ for all $s\in \mathcal{I}$.
\begin{proof}(of the lemma)
Recall \eqref{eq:lD} for the definition of $l_D(s)$.
In view of \eqref{eq:probability}, the probabilistic meaning of \eqref{eq:V-epsilon} is that $V_\epsilon(s)$ is attained when one chooses the excursion level $l_D(s)$ optimally in the following optimal stopping:
\begin{equation}\label{eq:new-osp}
  V_\epsilon(s)=\sup_{l_D(s)}\E^{s, s}[e^{-qT_{s+\epsilon}}\1_{\{T_{s+\epsilon}\le \tau_{s-l_D(s)}\}}V(s+\epsilon, s+\epsilon)+e^{-q\tau_{s-l_D(s)}}\1_{\{T_{s+\epsilon}> \tau_{s-l_D(s)}\}}(g-\bar{f})(s-l_D(s), s)],
\end{equation}
that is, if the excursion from $s$ does not reach the level of $l_D(s)$ before $X$ reaches $s+\epsilon$, one shall receive $V(s+\epsilon, s+\epsilon)$ and otherwise, one shall receive the reward. By using the transformation \eqref{eq:F-trans}, one needs to consider the function $\frac{(g-\bar{f})(x, s)}{\varphi(x)}$ and the point $\left(F(s+\epsilon), \frac{V(s+\epsilon, s+\epsilon)}{\varphi(s+\epsilon)}\right)$ in the $(F(x), z(x)/\varphi(x))$-plane. Then the value function of \eqref{eq:new-osp} in this plane is the smallest concave majorant of $\frac{(g-\bar{f})(x, s)}{\varphi(x)}$ which passes through the point $\left(F(s+\epsilon), \frac{V(s+\epsilon, s+\epsilon)}{\varphi(s+\epsilon)}\right)$. It follows that $\frac{V_\epsilon(s)}{\varphi(s)}\le  \frac{V(s+\epsilon)}{\varphi(s+\epsilon)}$.  As $\epsilon\downarrow 0$, it is clear that $\frac{V(s+\epsilon)}{\varphi(s+\epsilon)}\downarrow \frac{V(s, s)}{\varphi(s)}$ and $\frac{V_\epsilon(s)}{\varphi(s)}\downarrow \frac{V(s,s)}{\varphi(s)}$. Suppose, for  a contradiction, that we  have
\begin{equation}\label{eq:for-contradiction}
\alpha_s(\epsilon)\frac{V(s+\epsilon, s+\epsilon)}{\varphi(s+\epsilon)}<\frac{V_\epsilon(s)}{\varphi(s)}<\frac{V(s+\epsilon, s+\epsilon)}{\varphi(s+\epsilon)},
\end{equation} for all $\epsilon>0$.
This implies that the first term goes to $\frac{V(s,s)}{\varphi(s)}$ from below and the third term goes to the same limit from above.  While the second inequality always hold, the first inequality leads to a contradiction to the fact that the function $\epsilon\mapsto (1-\alpha_s(\epsilon))\frac{V(s+\epsilon, s+\epsilon)}{\varphi(s+\epsilon)}$ is continuous for all $s$.  Indeed, due to the monotonicity of $\alpha_s(\epsilon)\frac{V(s+\epsilon, s+\epsilon)}{\varphi(s+\epsilon)}$ in $\epsilon$, we would have $\frac{V_\epsilon(s)}{\varphi(s)}>\frac{V(s, s)}{\varphi(s)}>\alpha_s(\epsilon)\frac{V(s+\epsilon, s+\epsilon)}{\varphi(s+\epsilon)}$ for all $\epsilon>0$.  Hence one cannot make the distance between $\frac{V(s+\epsilon, s+\epsilon)}{\varphi(s+\epsilon)}$ and $\alpha_s(\epsilon)\frac{V(s+\epsilon, s+\epsilon)}{\varphi(s+\epsilon)}$ arbitrarily small without violating \eqref{eq:for-contradiction}.
This shows that there exists an $\epsilon'=\epsilon'(s)$ such that $\epsilon<\epsilon'$ implies that $\frac{V_\epsilon(s)}{\varphi(s)}\le \alpha_s(\epsilon)\frac{V(s+\epsilon, s+\epsilon)}{\varphi(s+\epsilon)}$.

On the other hand,  in \eqref{eq:new-osp}, one could choose a stopping time $\tau_{l_{D}(s)}$ that visits the left boundary $l$, then by reading \eqref{eq:probability} with $l_D(u)=u$ and $m=s+\epsilon$, \eqref{eq:new-osp} becomes
\[
V_\epsilon(s)\ge \frac{\varphi(s)}{\varphi(s+\epsilon)}\frac{F(s)}{F(s+\epsilon)}V(s+\epsilon, s+\epsilon)=\frac{\psi(s)}{\psi(s+\epsilon)}V(s+\epsilon, s+\epsilon)=\E^{s, s}(e^{-qT_{s+\epsilon}})V(s+\epsilon, s+\epsilon)
\] for \emph{any} $\epsilon>0$.
Since $s\in \mathcal{I}$ is a regular point, the last expectation can be arbitrarily close to unity (see page 89 \cite{IM1974}): that is, there exists an $\epsilon''=\epsilon''(s)>0$ such that $\epsilon<\epsilon''$ implies that  so that $V_\epsilon(s)\ge V(s+\epsilon, s+\epsilon)$.  By using the second assumption in the statement of Proposition, for any $s$, we have
$\frac{V_\epsilon(s)}{\varphi(s)}\ge \alpha_s(\epsilon)\frac{V(s+\epsilon, s+\epsilon)}{\varphi(s+\epsilon)}$ for $\epsilon<\epsilon''$. This completes the proof of Lemma \ref{lem:epsilon}.
\end{proof}

Let us continue the proof of Proposition \ref{prop:2}. By using \eqref{eq:V-V} in Lemma \ref{lem:epsilon}, we can write, for $\epsilon$ small,
\begin{equation}\label{eq:V_epsilon}
V_\epsilon(s)-\frac{\varphi(s)}{\varphi(s+\epsilon)}\exp\left(-\epsilon H(s; l^*_D)\right)V(s+\epsilon, s+\epsilon)=\left(1-\frac{\varphi'(s)}{\varphi'(s+\epsilon)}\exp\left(-\epsilon H(s; l^*_D)\right)\right)V_\epsilon(s).
\end{equation}
Moreover, since $\lim_{\epsilon\downarrow 0}V(s+\epsilon,s+\epsilon)=V(s,s)$, the optimal threshold $l^*_D(s)$ should satisfy
\begin{equation*}
V(s, s)=\lim_{\epsilon\downarrow 0}V_\epsilon(s)=\lim_{\epsilon\downarrow 0}\left[ \frac{\varphi(s)}{\varphi(s+\epsilon)}\exp\left(-\epsilon H(s; l^*_D)\right)V(s+\epsilon,s+\epsilon)
+\frac{\varphi(s)}{\varphi(s-l^*_D(s))}\cdot\epsilon H(s; l^*_D)G(s; l^*_D)\right],
\end{equation*}
from which equation, in view of \eqref{eq:V_epsilon}, we obtain
\begin{eqnarray*}
V(s,s)&=&\lim_{\epsilon\downarrow 0}\frac{ V_\epsilon(s)-\frac{\varphi(s)}{\varphi(s+\epsilon)}\exp\left(-\epsilon H(s; l^*_D)\right)V(s+\epsilon,s+\epsilon) }
{ 1-\frac{\varphi'(s)}{\varphi'(s+\epsilon)}\exp\left(-\epsilon H(s; l^*_D)\right) }\\
&=&\lim_{\epsilon\downarrow 0}\frac{ \frac{\varphi(s)}{\varphi(s-l^*_D(s))}\cdot\epsilon H(s; l^*_D)G(s; l^*_D) }
{ 1-\frac{\varphi'(s)}{\varphi'(s+\epsilon)}\exp\left(-\epsilon H(s; l^*_D)\right) }\\
&=&\frac{\varphi(s)}{\varphi(s-l^*_D(s))}
\frac{F'(s)\varphi'(s)}{\varphi''(s)[F(s)-F(s-l^*_D(s))]+F'(s)\varphi'(s)}
(g-\bar{f})(s-l^*_D(s), s),
\end{eqnarray*}
where the last equality is obtained by L'H\^{o}pital's rule, and hence $l^*_D(s)$ is the value which gives the supremum to $\frac{\varphi(s)}{\varphi(s-z)}
Q(s;z)
(g-\bar{f})(s-z, s)$.
\end{proof}

\begin{remark}\normalfont  Let us slightly abuse the notation by writing
$\frac{F'(s)\varphi'(s)}{\varphi''(s)[F(s)-F(s-z)]+F'(s)\varphi'(s)}=Q(s; z)$ to avoid the long expression.
Note that {\rm $\frac{\varphi(s)}{\varphi(s-z)}Q(s;z)(g-\bar{f})(s-z,s)$ is the value corresponding to the strategy $D$ with $l_D(s)=z$ and $l_D(m)=l^*_D(m)$ for every $m>s$; that is, this amount is obtained when we stop if $X$ goes below $s-z$ in the excursion at level $S=s$ and behave optimally at all  the higher levels $S>s$.}
\end{remark}
\begin{remark}\label{rem:shiryaev}\normalfont
Another representation of $V(s, s)$ is still possible. Continue to fix $s\in \II$.  The smooth-fit principle is assumed to hold at an optimal
point $s-l^*_D(s)$ and we have $(s-l_D^*(s), s)$ as a continuation region with the line $L_s$ tangent to the function $(g-\bar{f})(s-l_D(s), s)/\varphi(s-l^*_D(s))$ at $y=F(s-l^*_D(s))$.  Then we have a first-order approximation of $\frac{V(s+\epsilon, s+\epsilon)}{\varphi(s+\epsilon)}$:
    \begin{equation}\label{eq:approx}
    \frac{V(s+\epsilon, s+\epsilon)}{\varphi(s+\epsilon)}=\frac{V(s, s)}{\varphi(s)}+\gamma(s-l^*_D(s))\cdot \epsilon F'(s)
    \end{equation} where \begin{equation*}\label{eq:gamma}
    \gamma(s-l^*_D(s)):=\frac{1}{F'(s-l^*_D(s))}\left(\frac{(g-\bar{f})(s-l^*_D(s), s)}{\varphi(s-l^*_D(s))}\right)'\ge 0,
    \end{equation*}
    the slope of $L_s(y)$ at $F(s-l^*_D(s))$. See \eqref{eq:derivative}.
Equate equation \eqref{eq:approx} with \eqref{eq:V-epsilon-in-F} to obtain, for any $\epsilon$ sufficiently close to zero,
\[
\left(\frac{V_\epsilon(s)}{\varphi(s)}-\frac{\epsilon H(s; l^*_D)(g-\bar{f})(s- l^*_D(s), s)}{\varphi(s-l^*_D(s))}\right)e^{\epsilon H(s; l^*_D)}
=\frac{V(s, s)}{\varphi(s)}+\gamma(s-l^*_D(s))\cdot \epsilon F'(s).
\] Since $V_\epsilon(s)\mapsto V(s)$ as $\epsilon\rightarrow 0$,
\begin{align}\label{eq:for-special-V}
  \frac{V(s, s)}{\varphi(s)}&=\lim_{\epsilon\downarrow 0} \frac{\gamma(s-l^*_D(s))\cdot \epsilon F'(s)+\epsilon e^{\epsilon H(s; l^*_D)} H(s; l^*_D)\frac{(g-\bar{f})(s- l^*_D(s), s)}{\varphi(s-l^*_D(s))} }{e^{\epsilon H(s; l^*_D)}-1}\nonumber\\
  &=\frac{\gamma(s-l^*_D(s)) F'(s)+ H(s; l^*_D)\frac{(g-\bar{f})(s- l^*_D(s), s)}{\varphi(s-l^*_D(s))} }{H(s; l^*_D)}
\end{align}
where the last equality is obtained by L'H\^{o}pital's rule. By multiplying $\varphi(s)$ on both sides and performing some algebra, we have
\begin{align}\label{eq:V-explicit-2}
V(s,s)&=\frac{\varphi(s)}{H(s; l^*_D(s))}\left[\frac{F'(s)}{F'(s-l^*(D))}\left(\frac{(g-\bar{f})(s-l_D^*(s), s)}{\varphi(s-l^*_D(s))}\right)'+H(s; l^*_D(s))\frac{(g-\bar{f})(s-l_D^*(s), s)}{\varphi(s-l_D^*(s))}\right].
\end{align}
\begin{remark}[On the smooth-fit principle]\normalfont
We comment on the assumption of the smooth-fit in Remark \ref{rem:shiryaev}.  When the reward function $(g-\bar{f})(x, s)$ is a \emph{nondecreasing function of the second argument and is differentiable in both arguments},  we have $V(s+\epsilon, s+\epsilon)\ge V(s, s)$ and hence the approximation argument \eqref{eq:approx} is valid.  We shall see two examples in Section \ref{sec:special}.  Note that in Figure \ref{Fig-put}, we have $\frac{V(s, s)}{\varphi(s)}\ge \frac{(g-\bar{f})(s-l^*_D(s), s)}{\varphi(s-l^*_D(s))}$  and the smooth-fit holds at $s-l^*_D(s)$ in both cases.
\end{remark}
\end{remark}
\subsection{Special Cases}\label{sec:special}
Before moving on to find the general solution $V(x, s)$, it should be beneficial to briefly review some special cases in finding $V(s,s)$.  In this section, the diffusion $X$ is geometric Brownian motion $\diff X_t =\mu X_t\diff t+\sigma X_t\diff B_t$ and $(\A-q)v(x)=0$ provides $\varphi(x)=x^{\gamma_0}$ and $\psi(x)=x^{\gamma_1}$ with $\gamma_0<0$ and $\gamma_1>1$.  The parameters are $(\mu, \sigma, q, K, k)=(0.05, 0.25, 0.15, 5, 0.5)$. The values of the options here are computed under the physical measure $\p$.
\subsubsection{Perpetual Put} \label{sec:perpetual}
The reward function is $(g-\bar{f})(x, s)=g(x)=(K-x)^+$ which does not depend on $s$ and there is no absorbing boundary.  We can use Corollary \ref{cor:1}, for given $s$, to calculate $l^*_D(s)$ and the corresponding $V(s, s)$. Figure \ref{Fig-put}-(a) is the graph of $g(x, s)/\varphi(x)$ against the horizontal axis $y=F(x)$ when $s=5$.    The function
$R_s(x):=\frac{(g-\bar{f})(x, s)}{\varphi(x)}$ attains unique maximum at $F(x^*)$ where $x^*=3.57604$, so that $l^*_D(s)=s-x^*=1.42396$.  Since $g$ is independent of $s$, so is $x^*$.  At this point the tangent line has slope zero; that is, $\gamma(s-l^*_D(s))=0$.  See the red horizontal line connecting two points $(F(x^*), R_s(x))$ and $(F(s), V(s, s)/\varphi(s))$. At $F(x^*)$, we have the smooth-fit principle hold and $(x^*, s)$ is the continuation region.

Let us see the relationship with Remark \ref{rem:shiryaev}.  For this $l^*_D(s)$, as can be seen from the graph, $V(s+\epsilon)/\varphi(s+\epsilon)=V(s, s)/\varphi(s)$.  Then   \eqref{eq:V-explicit-2} (see also \eqref{eq:for-special-V}) reduces to a very simple form
\begin{equation*}
V(s, s)=\frac{\varphi(s)}{\varphi(s-l^*_D(s))}(g-\bar{f})(s-l^*_D(s), s),
\end{equation*}
which is the same as \eqref{eq:simple-case-no-s}.
\begin{figure}[h]
\begin{center}
\begin{minipage}{0.5\textwidth}
\centering{\includegraphics[scale=0.5]{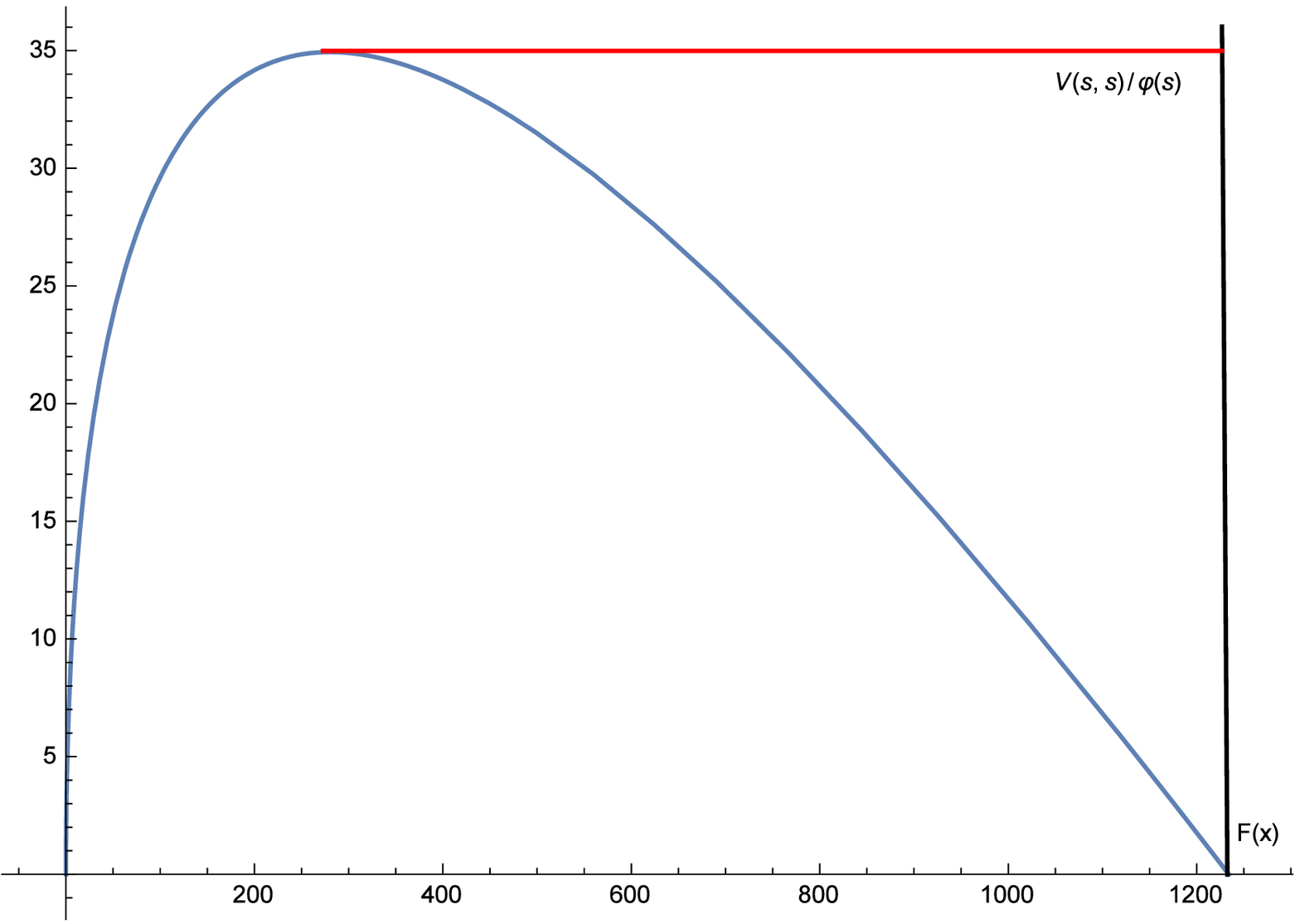}}\\
(a) Perpetual Put $(K=5)$
\end{minipage}
\begin{minipage}{0.45\textwidth}
\centering{\includegraphics[scale=0.4]{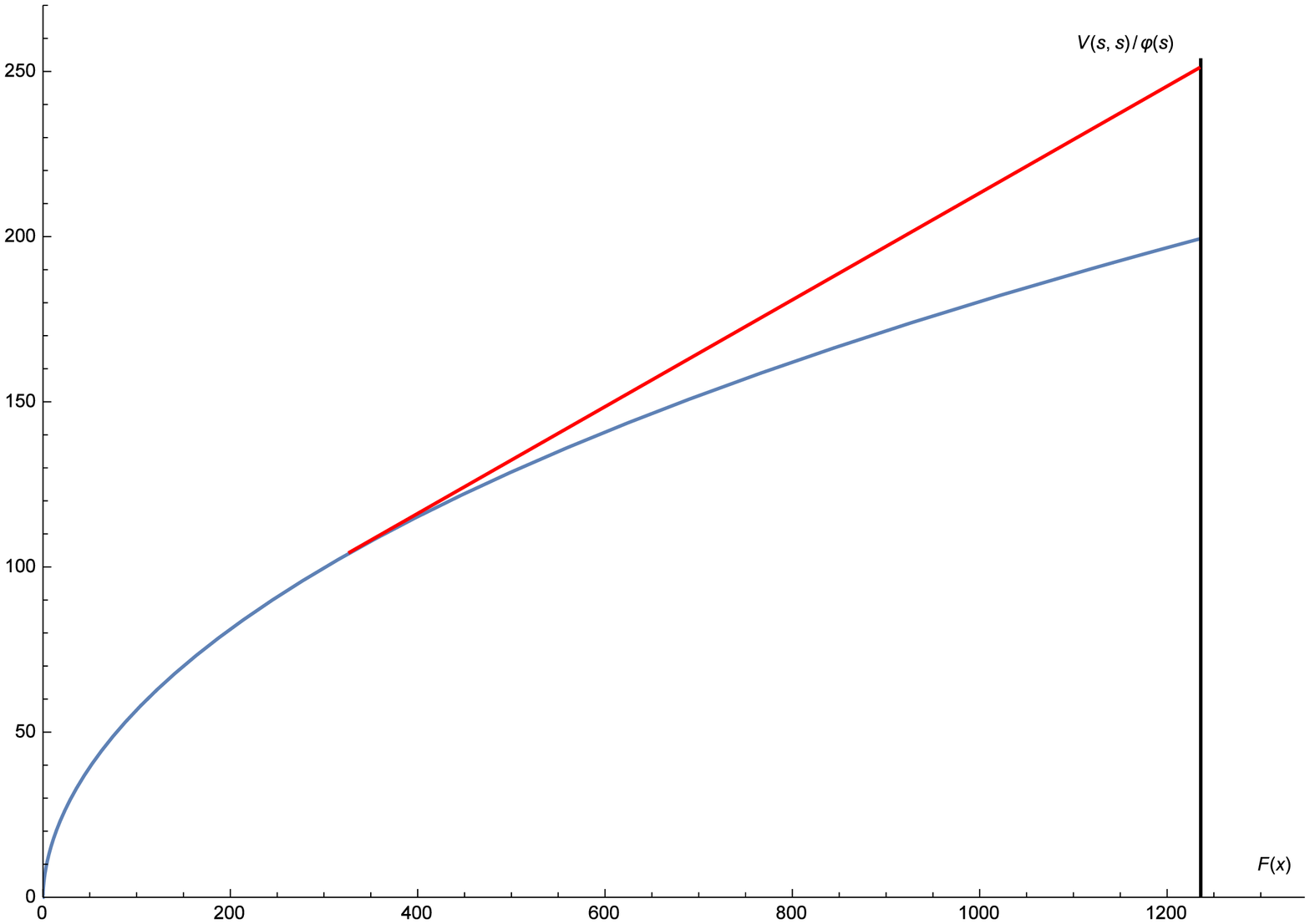}}\\
(b) Lookback Option $(k=1/2)$
\end{minipage}
\caption{\small \textbf{The graphs of $g(x, s)/\varphi(x)$ against the horizontal axis $F(x)$ : }  We fix $s=5$.  In the perpetual put case, the optimal exercise threshold is well-known: $x^*=\frac{\gamma_0K}{\gamma_0-1}=3.57604$, which does not depend on $s$. } \label{Fig-put}
\end{center}
\end{figure}
\subsubsection{Lookback Option}
The reward function is $(g-\bar{f})=s-kx$ where $k\in [0, 1]$. Set $s=5$.  The graph of $g(x, s)/\varphi(x)$ against the horizontal axis $F(x)$ is in Figure \ref{Fig-put}-(b). The optimal threshold $l^*_D(s)$ can be found by Proposition \ref{prop:2}: the optimal level $x^*$ is given by $x^*=\beta s$ where $\beta=0.784073$, independent of $s$, so that $l^*_D(s)=(1-\beta)/s$.

Once $l^*_D(s)$ is obtained, we can compute $V(s, s)$ from \eqref{eq:V-explicit}.  The red line $L_s$ is drawn connecting $(F(s-l^*_D(s)), R_s(s-l^*_D(s)))$ and $(F(s), V(s, s)/\varphi(s))$ with a positive slope $\gamma(s-l^*_D(s))$ and the smooth-fit principle holds at $F(s-l^*_D(s))$.  Accordingly, $(s-l^*_D(s), s)$ is the continuation region.

The case of $k=0$ was solved in the well-known \cite{shepp-shiryaev-1993}. For $k=0$, the reward $(g-\bar{f})(x, s)$ does not depend on $x$, so that we have further simplification of $V(s,s)$.  A straightforward computation from \eqref{eq:V-explicit-2} 
yields
\begin{align}\label{eq:simple-case-no-x}
  V(s,s)=\frac{\psi'(s-l^*)\varphi(s)-\psi(s)\varphi'(s-l^*)}{\psi'(s-l^*)\varphi(s-l^*)-\psi(s-l^*)\varphi'(s-l^*)}\cdot (g-\bar{f})(s)
\end{align}
where we write $l^*:=l^*_D(s)$ for simplicity.  It can be confirmed that \eqref{eq:simple-case-no-x} for $(g-\bar{f})(s)=s$ is the same as
\[
V(s, s)=\frac{s}{\gamma_1-\gamma_0}\left(\gamma_1\left(\frac{1}{\beta}\right)^{\gamma_0}-\gamma_0\left(\frac{1}{\beta}\right)^{\gamma_1}\right)
\] in \cite{shepp-shiryaev-1993}.

\subsection{General Solution}
Finally, let us consider the general case, $S_0\geq X_0$. Since we calculated $V(s,s)$, we can represent $V(x,s)$ by (\ref{eq:one-dim-version}):
\begin{eqnarray}\label{eq:gen}
V(x,s)&=&\sup_{\tau\in\S}\E^{x,s}\left[\1_{\{\tau<T_s\}}\1_{\{\tau<T^-_{s-b}\}}e^{-q\tau}(g-\bar{f})(X_{\tau},s)\right.\\
&&+\1_{\{T^-_{s-b}<T_s\}}\1_{\{T^-_{s-b}\leq \tau\}}e^{-qT^-_{s-b}}\{-\bar{f}(X_{T^-_{s-b}},s)\}\nonumber\\
&&\left.+\1_{\{T_s<\tau\wedge T^-_{s-b}\}}e^{-qT_s}V(s,s)\right].\nonumber
\end{eqnarray}
As we noted in Section 2, this can be seen as just an one-dimensional optimal stopping problem for the process $X$. So we can restrict the set $\S$ of stopping times to $S'(b)$ defined in (\ref{eq:D(m)}). However,
our problem is subject to absorption while $X$ is in its excursions and is different from ordinal ones.  In the last part of this section \ref{sec:method},
we shall illustrate how to implement the solution method presented in Dayanik and Karatzas \cite{DK2003}.
\begin{remark}\label{rem:pattern}\normalfont
  More general situations can be handled in our problem formulation. Recall that \eqref{eq:gen} is written under the assumption that when one hits the boundary, no reward would be given.  However, we could also assume that one  still obtains the reward in the amount of $g(X_\tau)$.  In this case, $-\bar{f}(X_{T^-_{s-b}},s)$ in the second term on the right-hand side of \eqref{eq:gen} should read $(g-\bar{f})(X_{T^-_{s-b}},s)$.
\end{remark}

\subsection{Solution Method}\label{sec:method}
Now suppose that we have found $V(s, s)$ for each $s\in \R_+$.  The next step is to solve \eqref{eq:gen}.  Consider an excursion from the level $S=s$. Recall that $V(s,s)$ represents the value that one would obtain when $X$ would return to that level $s$. On the other hand, if $X$ reaches the absorbing boundary before it returns to $s$, then one would obtain the amount of $(g-\bar{f})(X_{\zeta},s)$ under the assumption we receive $g(X_\zeta, s)$ at the absorbing boundary.

By recalling Section \ref{sec:diffusion-facts}, this fact translates into the following geometrical property: the value function in the transformed space must pass the points: \[\left(F(s-b(s)),\frac{(g-\bar{f})(s-b(s),s)}{\varphi(s-b(s))}\right) \conn \left(F(s),\frac{V(s,s)}{\varphi(s)}\right).\] Then the task is to find for each $s$ the smallest concave majorant $W_s(y)$ of $H_s(y):=\frac{(g-\bar{f})(F^{-1}(y),s)}{\varphi(F^{-1}(y))}$ and to recognize the region $H_s(y)=W_s(y)$ for optimal stopping region. 
Mathematically speaking, $W$ must satisfy the following conditions:

\begin{enumerate}
\renewcommand{\labelenumi}{(\roman{enumi})}
 \item $W_s(y)\geq\frac{(g-\bar{f})(y,s)}{\varphi(s-b(s))}$ on $[F(s-b(s)),F(s)]$,
 \item $W_s(F(s))=\frac{V(s,s)}{\varphi(s)}$,
 \item $W_s(F(s-b(s)))=\frac{(g-\bar{f})(s-b(s),s)}{\varphi(s-b(s))}$,
 \item $W_s$ is concave on $[F(s-b(s)),F(s)]$, and
 \item for any functions $\overline{W}_s$ which satisfies four conditions above, $W_s\leq\overline{W}_s$ on $[F(s-b(s)),F(s)]$.
\end{enumerate}
Now once we have done with one $s$, we then move on to another $\tilde{s}$, say, and find $W_{\tilde{s}}$ in the new interval
$[F(\tilde{s}-b(\tilde{s})), F(\tilde{s})]$.

Figure\ref{Fig1} illustrates a typical example of the graphs of $W_{s}$ and $H_{s}$ in transformed space.  Fix $s=\bar{s}$. Take two points $F(\s-b(\s))$ and $F(\s)$ on the horizontal axis and find $W_{\s}(y)$ that satisfies the above conditions. For this purpose, three dashed vertical lines are drawn at $y=F(\s-b(\s)), R_{\s}$, and $F(\s)$ from the left to right. $H_{\s}$ is the blue line on $y\in [F(\s-b(\s)),F(\s)]$. $W_{\s}$ is the green line on $y\in [F(\s-b(\s)),R_{\s}]$ and the blue line on $y\in(R_{\s},F(\s)]$. It follows that the point $(\s,\s)$ is included in the optimal stopping region.  Note that, in this example, $V(s,s)=(g-\bar{f})(s,s)$ holds for $s\ge s'$ such that $y=F(s')$ is the point where the red and blue line intersect in Figure \ref{Fig1}. In \eqref{eq:V-explicit}, this corresponds to $l^*_D(s')=0$. In other words, the region where the red line is drawn above the blue line indicates where $V(s, s)>(g-\bar{f})(s,s)$. Therefore, if it happens that some $s''$ falls into this region, optimal strategy for this particular level $s''$ is to wait until $X$ moves back to $s''$ (and continue) or hits the boundary $s''-b(s'')$.
\begin{figure}[h]
\begin{center}
\begin{minipage}{0.6\textwidth}
\centering{\includegraphics[scale=1]{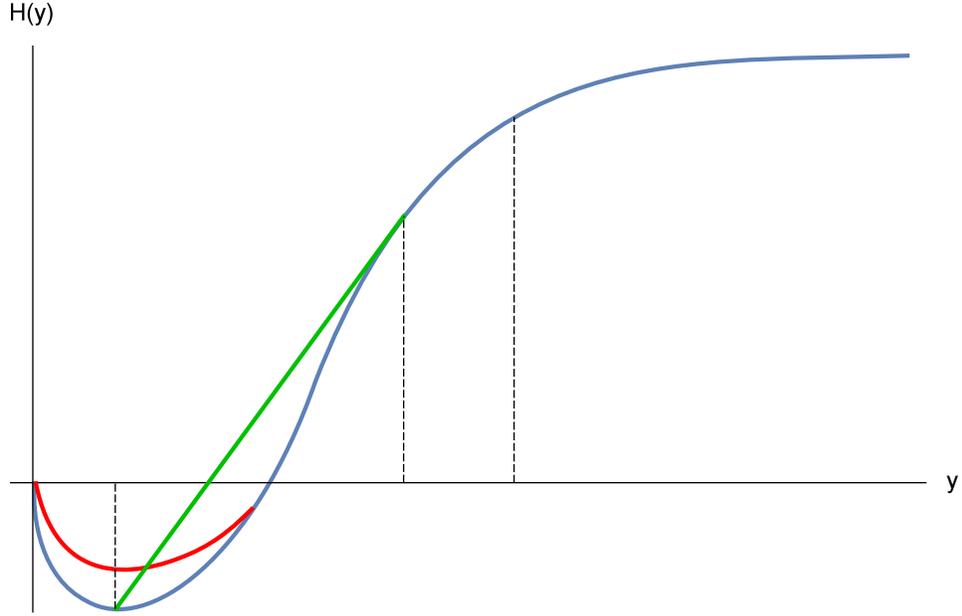}}\\
\end{minipage}
\caption{A typical example of $W_{s}$ and $H_{s}$. For a fixed $s=\s$, three vertical dashed lines are drawn at $y=F(\s-b(\s)), R_{\s}, F(\s)$ from the left to right. $H_{\s}$ is the blue line on $y\in [F(\s-b(\s)),F(\s)]$. $W_{\s}$ is the green line on $y\in [F(\s-b(\s)),R_{\s}]$ and the blue line on $y\in(R_{\s},F(\s)]$.} \label{Fig1}
\end{center}
\end{figure}
\section{Application}
In Section \ref{sec:special}, we have demonstrated how to find $V(s, s)$ when the reward function $(g-\bar{f})(x, s)$ depends on $s$ or not, we now focus on illustrating the way one can find the general $V(x, s)$ and how the absorbing boundary $b=b(S)$ affects the solution.  Specifically, we solve the following financial problem. Imagine that an investor considers investing in a hedge fund that manages a risky portfolio.  The investor is conservative enough to take the strategy that when her asset value deteriorates to $100\beta\%$ ($\beta\in[0,1)$) of the maximum record, then she sells all her stakes in the hedge fund and takes the proceeds into risk-free assets. The problem is when she should stop the investment in the hedge fund.

We assume that a geometric Brownian motion $X$ represents the asset value in the risky portfolio and satisfies the following stochastic differential equation:
\[
\diff X_t=\mu X_t \diff t + \sigma X_t \diff B_t,\quad t\in \R,
\]
where $\mu$ and $\sigma$ are constants and $B$ is a standard Brownian motion under $\mathbb{P}$. We consider that the risk-free rate is $q$,  the reward functions in our problem $f$ and $g$ are set to $f(x,s)=x^{\frac{1}{2}}$ and $g(x,s)=x$, and the absorbing boundary is $b(s)=(1-\beta) s$ to match the problem explained above. In addition, we assume
\[
\frac{\mu}{2}-\frac{\sigma^2}{8}-q<0,
\]
for the convergence of the continuous reward part.

Under those settings, related functions are calculated as follows:
\begin{align*}
\bar{f}(x)&=\E^{x,s} \left[ \int^{\infty}_0 e^{-qt}f(X_t,S_t)\diff t \right]=\alpha x^{\frac{1}{2}},\\
\psi(x)& =x^{\gamma_1} \quad\text{and}\quad \varphi(x)=x^{\gamma_0}
\end{align*}
where
\begin{eqnarray*}
\alpha&=&-\frac{1}{\frac{\mu}{2}-\frac{\sigma^2}{8}-q}>0,\\
\gamma_0&=&\frac{1}{2}\left(-\left(\frac{2\mu}{\sigma^2}-1\right)-\sqrt{\left(\frac{2\mu}{\sigma^2}-1\right)^2+\frac{8q}{\sigma^2}}\right)<0,
\end{eqnarray*}
and
\begin{eqnarray*}
\gamma_1&=&\frac{1}{2}\left(-\left(\frac{2\mu}{\sigma^2}-1\right)+\sqrt{\left(\frac{2\mu}{\sigma^2}-1\right)^2+\frac{8q}{\sigma^2}}\right)>1.
\end{eqnarray*}

Now we identify $W$. 
$V(s,s)$ can be computed from Corollary \ref{cor:1}, which is
\begin{equation}\label{Vss}
V(s,s)=
\begin{cases}
\beta^{1-\gamma_0}s-\alpha\beta^{\frac{1}{2}-\gamma_0}s^{\frac{1}{2}},& \text{ on }\quad s<\left( \frac{ \alpha(1-\beta^{\frac{1}{2}-\gamma_0}) }{ 1-\beta^{1-\gamma_0} } \right),\\
(g-\bar{f})(s,s),& \text{ on }\quad s\geq\left( \frac{ \alpha(1-\beta^{\frac{1}{2}-\gamma_0}) }{ 1-\beta^{1-\gamma_0} } \right).
\end{cases}
\end{equation}
One the other hand, the reward function $H_s$ in the transformed space is
\[
H(y)=\frac{(g-\bar{f})(F^{-1}(y),s)}{\varphi(F^{-1}(y))}=y^{\frac{1-\gamma_0}{\gamma_1-\gamma_0}}-\alpha y^{\frac{\frac{1}{2}-\gamma_0}{\gamma_1-\gamma_0}}.
\]
$H_s$ has the inflection point at $y=r$, where
\begin{equation}\label{infection}
r:=\left(\frac{\alpha(\frac{1}{2}-\gamma_0)(\frac{1}{2}-\gamma_1)}{(1-\gamma_0)(1-\gamma_1)}\right)^{2(\gamma_1-\gamma_0)},
\end{equation}
so $H$ is convex on $y<r$ and concave on $y>r$ and goes to infinity as $y$ goes to infinity, so that $H(y)$ looks like the blue curve in Figure \ref{Fig1}. The tangent line $L_s$ to $H_s$ at point $(F(s),\frac{V(s,s)}{\varphi(s)})$, which is represented by
\[
L_s(y)=\left(\frac{1-\gamma_0}{\gamma_1-\gamma_0}s^{1-\gamma_1}-\alpha\frac{\frac{1}{2}-\gamma_0}{\gamma_1-\gamma_0}s^{1-\gamma_1}\right)y-\left(\frac{1-\gamma_0}{\gamma_1-\gamma_0}-1\right)s^{\frac{1}{2}-\gamma_1}+\alpha\left(\frac{\frac{1}{2}-\gamma_0}{\gamma_1-\gamma_0}-1\right)s^{\frac{1}{2}-\gamma_1}.
\]
So we compare $H_s(y)$ and $L(y)$ at point $y=s-b(s)$ to find
\begin{equation}\label{tangent}
\begin{cases}
L_s(s-b(s))\leq H_s(s-b(s)),& \text{ on } \quad F(s)\leq u,\\
L_s(s-b(s))> H_s(s-b(s)),& \text{ on } \quad F(s)>u,
\end{cases}
\end{equation}
where
\[
u=\left( \frac{ \alpha\left(\frac{\frac{1}{2}-\gamma_0}{\gamma_1-\gamma_0}+\beta^{\frac{\frac{1}{2}-\gamma_0}{\gamma_1-\gamma_0}}-\beta^{\gamma_1-\gamma_0}\frac{\frac{1}{2}-\gamma_0}{\gamma_1-\gamma_0}-1\right) }{ \frac{1-\gamma_0}{\gamma_1-\gamma_0}+\beta^{\frac{1-\gamma_0}{\gamma_1-\gamma_0}}-\beta^{\gamma_1-\gamma_0}\frac{1-\gamma_0}{\gamma_1-\gamma_0}-1 } \right)^{2(\gamma_1-\gamma_0)}
\]

From (\ref{Vss}), (\ref{infection}), and (\ref{tangent}), we can identify the function $W$. We divide into three cases depending on the value of $s$.

(i) If \emph{$F(s)\leq u$}, then $W_s(y)$ is the linear function connecting two points $(F(s-b(s)),\frac{(g-\bar{f})(s-b(s),s)}{\varphi(s-b(s))})$ and $(F(s),\frac{V(s,s)}{\varphi(s)})$, that is,
\[
W_s(y)=\frac{\frac{V(s,s)}{\varphi(s)} - \frac{(g-\bar{f})(s-b(s),s)}{\varphi(s-b(s))}}{F(s)-F(s-b(s))}(y-F(s))+\frac{V(s,s)}{\varphi(s)}.
\]
The optimal strategy is to wait until $X$ reaches to $s$ or $s-b(s)$.\\

(ii) If \emph{$F(s-b(s))\leq r$ and $u<F(s)$}, draw a tangent line from $(F(s-b(s)),\frac{(g-\bar{f})(s-b(s),s)}{\varphi(s-b(s))})$ to $H$, and let $y=R_s$ be the tangent point of this line and $H$. Then $W_s(y)$ can be represented as follows:
\[
W_s(y)=
\begin{cases}
\frac{\frac{(g-\bar{f})(R_s,s)}{\varphi(R_s)} - \frac{(g-\bar{f})(s-b(s),s)}{\varphi(s-b(s))}}{F(R_s)-F(s-b(s))}(y-F(s-b(s)))+\frac{(g-\bar{f})(s-b(s),s)}{\varphi(s-b(s))}, &\text{on} \quad y\in[F(s-b(s)), R_s),\\
H_s(y), &\text{on} \quad y\in[R_s,F(s)].
\end{cases}
\]
This implies that when $X_0\in(s-b(s),F^{-1}(R_s))$, one should wait until $X$ hits $s-b(s)$ or $F^{-1}(R_s)$, and when $X_0\in[F^{-1}(R_s),s]$, one should stop immediately.
\\

(iii) If \emph{$F(s-b(s))>r$}, then $H_s$ is concave on $[F(s-b(s)),F(s)]$, so we have
\[
W_s(y)=H_s(y)=y^{\frac{1-\gamma_0}{\gamma_1-\gamma_0}}-\alpha y^{\frac{\frac{1}{2}-\gamma_0}{\gamma_1-\gamma_0}},
\]
and optimal strategy is to stop immediately to receive $g-\bar{f}$. Note that this $W_s(\cdot)$ does not depend on $s$ because nor does $g-\bar{f}$.

Figure \ref{Fig2} illustrates an optimal stopping region on $(x,s)$-plane for this problem. When $(X,S)$ reaches to the shaded region, one should stop the process. Note that $(X,S)$ moves in the upward direction only along the line $x=s$ and moves horizontally in other parts.

\begin{figure}[h]
\begin{center}
\begin{minipage}{0.6\textwidth}
\centering{\includegraphics[scale=1]{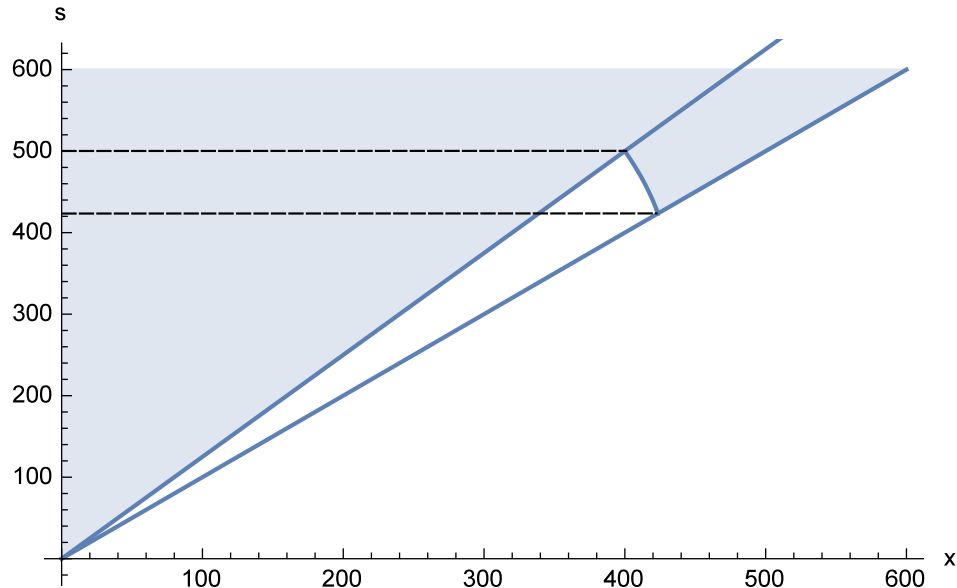}}\\
\end{minipage}
\caption{The shaded region is an optimal stopping region on $(x,s)$-plane. The left straight line is $s=x/\beta$ for $\beta=0.8$ and the right one is $s=x$. Two horizontal dashed lines are drawn at $s=F^{-1}(u),F^{-1}(r)/\beta$. The parameters are $(\mu, \sigma, q)=(0.05, 0.1, 0.1)$} \label{Fig2}
\end{center}
\end{figure}

\bibliographystyle{plain}
{\small \bibliography{ospbib2}}

\def\cprime{$'$}
\begin{thebibliography}{10}

\bibitem{alvarez2}
L.~H.~R. Alvarez.
\newblock On the properties of r-excessive mappings for a class of diffusions.
\newblock {\em Ann. Appl. Probab.}, 13 (4):1517--1533, 2003.

\bibitem{alvarez-matomaki2014}
L.~H.~R. Alvarez and P.~Matom\"{a}ki.
\newblock Optimal stopping of the maximum process.
\newblock {\em J. Appl. Prob.}, 51:818--836, 2014.

\bibitem{avram-et-al-2004}
F.~Avram, A.~E. Kyprianou, and M.~R. Pistorius.
\newblock Exit problems for spectrally negative {L}\'{e}vy processes and
  applications to ({C}anadized) {R}ussion options.
\newblock {\em Ann. Appl. Probab.}, 14:215--235, 2004.

\bibitem{Bertoin_1996}
J.~Bertoin.
\newblock {\em {L}\'evy processes}, volume 121 of {\em Cambridge Tracts in
  Mathematics}.
\newblock Cambridge University Press, Cambridge, 1996.

\bibitem{borodina-salminen}
A.~N. Borodin and P.~Salminen.
\newblock {\em Handbook of Brownian Motion - Facts and Formulae, Second
  Edition}.
\newblock Birkh\"{a}user, Basel, Boston, Berlin, 2002.

\bibitem{cinlar}
E.~\c{C}inlar.
\newblock {\em Probability and Stochastics}.
\newblock Springer, 2011.

\bibitem{DK2003}
S.~Dayanik and I.~Karatzas.
\newblock On the optimal stopping problem for one-dimensional diffusions.
\newblock {\em Stochastic Process. Appl.}, 107 (2):173--212, 2003.

\bibitem{Doney_2005}
R.~A. Doney.
\newblock Some excursion calculations for spectrally one-sided {L}\'evy
  processes.
\newblock In {\em S\'eminaire de {P}robabilit\'es {XXXVIII}}, volume 1857 of
  {\em Lecture Notes in Math.}, pages 5--15. Springer, Berlin, 2005.

\bibitem{dynkin}
E.~B. Dynkin.
\newblock {\em Markov Processes II}.
\newblock Springer, Berlin Heidelberg, 1965.

\bibitem{Egami-Oryu2013}
M.~Egami and T.~Oryu.
\newblock Optimal stopping when the absorbing boundary is following after.
\newblock {\em Working Paper; Kyoto University,
  http://www.econ.kyoto-u.ac.jp/~egami/two-dim.pdf}, 2014.

\bibitem{Egami-Oryu2013a}
M.~Egami and T.~Oryu.
\newblock An excursion-theoretic approach to regulator's bank reorganization
  problem.
\newblock {\em Oper. Res.}, 63:527--539, 2015.

\bibitem{Guo-Zervos_2010}
X.~Guo and M.~Zervos.
\newblock $\pi$ options.
\newblock {\em Stoch. Proc. Appl.}, 120:1033--1059, 2007.

\bibitem{IM1974}
K.~It\^{o} and H.~P.~McKean Jr.
\newblock {\em Diffusion Processes and their Sample Paths}, volume~6.
\newblock Springer, Berlin Heidelberg, 1974.

\bibitem{ott_2013}
C.~Ott.
\newblock Optimal stopping problems for the maximum process with upper and
  lower caps.
\newblock {\em Ann. Appl. Probab.}, 23:2327--2356, 2013.

\bibitem{peskir1998}
G.~Peskir.
\newblock Optimal stopping of the maximum process: the maximality principle.
\newblock {\em Ann. Probab.}, 26:1614--1640, 1998.

\bibitem{Pham-book}
H.~Pham.
\newblock {\em Continuous-time Stochastic Control and Optimization with
  Financial Applications}, volume~61 of {\em Stochastic Modeling and Applied
  Probability}.
\newblock Springer, Berlin Heidelberg, 2009.

\bibitem{Pistorius_2004}
M.~R. Pistorius.
\newblock On exit and ergodicity of the spectrally one-sided {L}\'evy process
  reflected at its infimum.
\newblock {\em J. Theoretical Probab.}, 17:183--220, 2004.

\bibitem{Pistorius_2007}
M.~R. Pistorius.
\newblock An excursion-theoretical approach to some boundary crossing problems
  and the {S}korokhod embedding for reflected {L}\'evy processes.
\newblock In {\em S\'eminaire de {P}robabilit\'es {XL}}, volume 1899 of {\em
  Lecture Notes in Math.}, pages 278--307. Springer, Berlin, 2007.

\bibitem{shepp-shiryaev-1993}
L.~Shepp and A.~N. Shiryaev.
\newblock The {R}ussian otpion: reduced regret.
\newblock {\em Ann. Appl. Probab.}, 3:631--640, 1993.

\end{thebibliography}

\end{document}